\documentclass[10pt,reqno,twoside]{amsart}

\usepackage[english]{babel}

\usepackage[T1]{fontenc}
\usepackage[utf8]{inputenc}

\usepackage[margin=1in]{geometry}

\usepackage{amsmath,amssymb,amsthm}
\usepackage{mathtools,mathrsfs}

\usepackage{theoremref}

\numberwithin{equation}{section} \numberwithin{figure}{section} 

\theoremstyle{plain} 
\newtheorem{thm}{\protect\theoremname}
\newtheorem{prop}[thm]{\protect\propositionname} \newtheorem{lem}[thm]{\protect\lemmaname} \newtheorem{cor}[thm]{\protect\Corollaryname}  

\theoremstyle{definition} 
 
   \newtheorem{rem}{\protect\remarkname} 

\makeatother

\usepackage{microtype} 

\providecommand{\theoremname}{Theorem}
\providecommand{\propositionname}{Proposition}
\providecommand{\lemmaname}{Lemma}
\providecommand{\Corollaryname}{Corollary}
\providecommand{\questionname}{Question}
\providecommand{\defnname}{Definition}
\providecommand{\exname}{Example}
\providecommand{\notationname}{Notation}
\providecommand{\Conventionname}{Convention}
\providecommand{\remarkname}{Remark}

\theoremstyle{plain} 
\newtheorem{mainthm}{Theorem} 

\usepackage{graphicx}
\usepackage[all]{xy}
\usepackage{tikz}
\usepackage{tikz-cd}
\usepackage{tkz-tab}
\usetikzlibrary{decorations.markings}
\usepackage[new]{old-arrows}

\usepackage{paralist}
\usepackage{subfigure}
\usepackage{multicol}
\usepackage{comment}
\usepackage{placeins}

\usepackage{xcolor}

\definecolor{Chocolat}{rgb}{0.36,0.20,0.09}
\definecolor{BleuTresFonce}{rgb}{0.215,0.215,0.86}

\usepackage[
colorlinks,
final,
hyperindex,
pagebackref=true
]{hyperref}

\hypersetup{
    linkcolor=blue,
    citecolor=purple,
    filecolor=BleuTresFonce,
    urlcolor=BleuTresFonce
}
\usepackage[capitalise]{cleveref}

\newcommand{\A}{\mathbf{A}}
\newcommand{\C}{\mathbf{C}}

\newcommand{\G}{\mathbb{G}}
\newcommand{\N}{\mathbf{N}}
\newcommand{\Q}{\mathbf{Q}}

\newcommand{\Z}{\mathbf{Z}}

\newcommand{\PP}{\mathbb{P}}

\DeclareMathOperator{\Aut}{Aut}

\DeclareMathOperator{\Spec}{Spec}

\DeclareMathOperator{\Cl}{Cl}
\DeclareMathOperator{\Pic}{Pic}

\DeclareMathOperator{\rank}{rank}

\DeclareMathOperator{\id}{id}

\tikzstyle{vertex}=[circle,draw,inner sep=0pt,minimum size=5pt]

\begin{document}
\title[Smooth affine surfaces properly dominated by $\C^*\times\C^*$]{Smooth affine surfaces properly dominated by $\C^*\times\C^*$}

\author[B. Hajra]{Buddhadev Hajra}
\address{Stat-Math Unit, Indian Statistical Institute Kolkata,
203 B.~T.~Road, Baranagar, Kolkata 700108, India}
\email{hajrabuddhadev92@gmail.com}



\begin{abstract}
We classify smooth complex affine surfaces admitting a finite surjective
morphism from $\C^*\times\C^*$. Using the
classification theory of smooth affine surfaces according to logarithmic
Kodaira dimension, we show that every such surface has logarithmic Kodaira
dimension either $-\infty$ or $0$, and determine all possibilities. More
precisely, the only surfaces of logarithmic Kodaira dimension $-\infty$ are
$\C^2$ and $\C\times\C^*$, while those of logarithmic Kodaira dimension $0$
are precisely $\C^*\times\C^*$ and Fujita's surface $H[-1,0,-1]$. This
establishes the classification of smooth affine surfaces properly dominated by
$\C^*\times\C^*$ anticipated by M.~Furushima in 1989.
\end{abstract}

\maketitle

\tableofcontents

\section{\bf Introduction}\label{se1}

Throughout, all algebraic varieties and morphisms are defined over the field of complex numbers $\C$.

\medskip

An algebraic variety $V$ is said to be \emph{properly dominated} by an algebraic variety $W$ if there exists a proper surjective morphism $\varphi\colon W\to V$. Since proper morphisms between affine varieties are finite, if $V$ and $W$ are affine, this is equivalent to requiring that $\varphi$ be finite and surjective. Equivalently, writing $V=\Spec A$ and $W=\Spec B$, where $A$ and $B$ are finitely generated $\C$-algebras, proper domination is equivalent to $B$ being integral (or, equivalently, finite) over $A$.

A fundamental source of examples is provided by Noether normalization, which asserts that every affine variety of dimension $d$ admits a finite surjective morphism onto the affine space $\A^d$. Thus every affine variety properly dominates the affine space of the same dimension.

The problem becomes considerably subtler when the dominating variety is fixed and one seeks to determine the resulting geometric or structural restrictions on the target. This question has long been an important theme in affine algebraic geometry. For affine surfaces, it is closely related to affine fibrations, the logarithmic Kodaira dimension, and the geometry of the boundary divisor in a suitable log-smooth completion.

Normal affine surfaces properly dominated by $\C^2$ have been extensively studied by Miyanishi, Gurjar, Shastri, Furushima, and others (see, for example, \cite{Miy1980b,Miy1986} for algebro-geometric proofs and \cite{Gur1980a,GS1984,GS1985} for topological proofs). Furushima subsequently investigated an analytic analogue of this problem by replacing finite surjective algebraic morphisms by proper holomorphic maps from $\C^2$ to normal complex analytic surfaces admitting analytic compactifications\footnote{A normal complex analytic surface $X$ is said to admit an analytic compactification if there exist a normal compact complex surface $\bar{X}$ and an analytic subset $Z\subseteq\bar{X}$ such that $\bar{X}\setminus Z$ is biholomorphic to $X$.} \cite{Fur1986}. He later studied normal affine surfaces properly dominated by $\C\times\C^*$ \cite{Fur1989}. The next natural problem is therefore to classify affine surfaces properly dominated by $\C^*\times\C^*$.

Furushima's analysis of the $\C\times\C^*$ case relies heavily on the Nishino--Suzuki theorem on cluster sets of holomorphic mappings from a punctured disc into a smooth compact complex surface, together with Suzuki's study of analytic compactifications of $\C\times\C^*$. Using these techniques, Furushima remarked that analogous methods should yield a classification of normal affine surfaces properly dominated by $\C^*\times\C^*$, namely
\[
\C^2,\qquad
\C\times\C^*,\qquad
\C^*\times\C^*,
\]
and suitable finite quotients of these surfaces by subgroups of $GL(2,\C)$ (see \cite[Remark]{Fur1989}). Although Furushima indicated that the details would appear elsewhere, to the best of the author's knowledge no complete proof has subsequently appeared in the literature.

The objective of the present article is to establish the classification of smooth affine surfaces properly dominated by $\C^*\times\C^*$, thereby confirming Furushima's prediction in the smooth case. Our approach is entirely different from Furushima's analytic methods. Instead, it relies on recent advances in the theory of open algebraic surfaces, most notably the classification of smooth non-contractible affine Eilenberg--MacLane $K(\pi,1)$-surfaces established in \cite[Result~2.1]{GGH2023}.

\medskip

The main result of this article is the following.

\begin{mainthm}[Theorems~\ref{Thm: Affine surfaces with kappa negative properly dominated by 2-torus} and \ref{Thm: Affine surfaces with kappa=0 properly dominated by 2-torus}]
\label{thm:mainB}
Let $Y$ be a smooth affine surface properly dominated by $\C^*\times\C^*$. Then $\bar\kappa(Y)\in\{-\infty,0\}$, and the following assertions hold.

\begin{enumerate}[\rm(1)]
\item If $\bar\kappa(Y)=-\infty$, then
\[
Y\cong\C^2
\quad\text{or}\quad
Y\cong\C\times\C^*.
\]

\item If $\bar\kappa(Y)=0$, then
\[
Y\cong\C^*\times\C^*
\quad\text{or}\quad
Y\cong H[-1,0,-1].
\]
\end{enumerate}
\end{mainthm}

The appearance of the surface $H[-1,0,-1]$, introduced by Fujita \cite{Fuj1982}, is particularly noteworthy. It was proved in \cite[Proposition~5.8]{GGH2023} that $H[-1,0,-1]$ admits a finite \'etale double cover
\[
\pi\colon\C^*\times\C^*\to H[-1,0,-1].
\]
Since every connected double cover is Galois, the deck transformation group $\operatorname{Deck}(\pi)$ is isomorphic to $\Z/2\Z$. Hence there exists a fixed-point-free involution $\sigma\in\Aut(\C^*\times\C^*)$ such that
\[
\operatorname{Deck}(\pi)=\langle\sigma\rangle
\quad\text{and}\quad
H[-1,0,-1]\cong(\C^*\times\C^*)/\langle\sigma\rangle.
\]
Thus Theorem~\ref{thm:mainB} confirms Furushima's predicted classification in the smooth setting.

\medskip

The paper is organized as follows. In Section~\ref{se2}, we collect the preliminary material used throughout the paper. Besides fixing notation and conventions, we recall the Suzuki--Zaidenberg formula for the topological Euler characteristic of smooth affine surfaces admitting affine fibrations, several auxiliary results on affine fibrations and finite morphisms, the theory of the fundamental group at infinity of normal algebraic surfaces due to Mumford and its extension by Wagreich, and the class $\mathscr{S}_0$ of smooth factorial affine surfaces of logarithmic Kodaira dimension zero with trivial units. Section~\ref{se3} contains the proofs of the main results. We first establish the case $\bar\kappa=-\infty$ by proving Theorem~\ref{Thm: Affine surfaces with kappa negative properly dominated by 2-torus}. We then develop several auxiliary results on finite \'etale covers of complements of finite subsets of $\C^*\times\C^*$ and on certain classes of affine surfaces that cannot be properly dominated by $\C^*\times\C^*$. These ingredients are combined in the final part of the section to prove Theorem~\ref{Thm: Affine surfaces with kappa=0 properly dominated by 2-torus}.
\section{\bf Preliminaries}\label{se2}

We begin by fixing notation and conventions that will be used throughout the article.

\subsection{Notations \& Conventions}

\begin{enumerate}[\indent$\bullet$]

\item The $n$-dimensional affine and projective spaces over $\C$ are denoted
by $\A^n$ (or sometimes by $\C^n$) and $\PP^n$, respectively, for all
$n\in\Z_{\ge0}$.

\item For a complex algebraic variety $X$, we denote by $X^{\rm an}$ its
associated complex analytic space equipped with the usual complex Euclidean topology whenever it is necessary to distinguish
between the algebraic and analytic categories.

\item We write $\G_m=\Spec\C[t^{\pm1}]$ for the algebraic one-dimensional
torus over $\C$ and $\G_m^2=\Spec\C[t^{\pm1},u^{\pm1}]$ for the algebraic
two-dimensional torus over $\C$. Their associated complex analytic spaces
are canonically identified with $\C^*$ and $\C^*\times\C^*$, respectively.
Following the standard convention in the literature, we shall often use the
notation $\C^*$ in place of $\G_m$ whenever no confusion between the
algebraic and analytic categories can arise.

\item If $X$ is a simplicial complex or a finite CW complex, then for
$i\in\Z_{\ge0}$ we denote by $\pi_1(X)$, $H_i(X;G)$ (resp.~$H^i(X;G)$) the
fundamental group, the $i$-th homology (resp.~cohomology) group of $X$ with
coefficients in $G$, respectively.

\item Unless explicitly stated otherwise, when $X$ is a complex algebraic variety, the invariants $\pi_1(X)$, $H_i(X;G)$ and $H^i(X;G)$ refer to the corresponding topological invariants of the associated complex analytic space $X^{\rm an}$.

\item For a smooth non-complete surface $S$, we use the notation
\[
\bar\kappa(S): \text{ the logarithmic Kodaira dimension of }S.
\]
For the definition of $\bar\kappa$, see~\cite{Iit1982}.

\item For a smooth complex algebraic variety $X$, we adopt the notation
\[
b_i(X)\;(i\in\Z_{>0}): \text{ the $i$-th Betti number }
\dim_{\C}H_i(X;\C),
\]
\[
e(X): \text{ the Euler--Poincar\'e characteristic of }X,\text{ defined by }
e(X)=\sum_{i=0}^{\infty}(-1)^ib_i(X).
\]

\item For a smooth complex affine variety $X$, we reserve the notation
\[
\rho(X):=\rank_{\Z}\Pic(X)
=\dim_{\Q}\bigl(\Pic(X)\otimes_{\Z}\Q\bigr).
\]

\end{enumerate}

\bigskip

We now recall some well-known results that will frequently be used in the subsequent sections.

\subsection{Suzuki--Zaidenberg formula for the topological Euler-Poincar{\'e} characteristic}

M.~Suzuki proved an important formula for the Euler--Poincar\'e characteristic of a smooth affine surface fibered over a smooth algebraic curve (cf.~\cite{Suz1977}).

\begin{thm}[{Suzuki--Zaidenberg Formula; cf.~\cite{Suz1977, Zai1987}}]
\label{Suzuki's formula}
Let $f\colon V\to C$ be a surjective morphism from a smooth affine surface $V$ onto a smooth algebraic curve $C$ such that a general fiber $F$ of $f$ is irreducible. Let $p_1,\ldots,p_n$ be the points of $C$ at which $f$ is not $\mathscr{C}^{\infty}$-locally trivial. Then
\[
e(V)=e(C)\cdot e(F)+\sum_{i=1}^{n}\bigl(e(F_i)-e(F)\bigr),
\]
where $F_i:=f^{-1}(p_i)$ for $1\le i\le n$.

Moreover, each summand $e(F_i)-e(F)$ is non-negative. If $e(F_i)=e(F)$ for some $i$, then $F$ is isomorphic to either $\C$ or $\C^{\ast}$, and $(F_i)_{\mathrm{red}}$ is isomorphic to $F$.
\end{thm}

Here, the last part about the non-negativity of every summand appearing in the above formula was essentially due to M.~Zaidenberg (cf. \cite{Zai1987}). The proofs in \cite{Suz1977} and \cite{Zai1987} use the theory of plurisubharmonic functions. A more algebro-geometric approach to the proof of this result was later given by R.~V.~Gurjar (cf.~\cite{Gur1997}).

\subsection{Some useful assorted results}

The following result is well-known. For the sake of completeness, we add a short proof. This result will be used in subsequent proofs.

\begin{lem}\label{Lem: Q-factorial fibration}
Let $f\colon X\to B$ be an $F$-fibration from a smooth affine surface $X$ onto a smooth algebraic curve $B$, where $F$ is a smooth affine rational curve. If $\rho(X)=0$, then the following holds.
\begin{enumerate}[\indent\rm(1)]
\item $B$ is affine.
\item All fibers of $f$ are irreducible.
\end{enumerate}
\end{lem}

\begin{proof}
Let $f \colon X \to B$ be an $F$-fibration from a smooth affine surface $X$ with $\rho(X)=0$ onto a smooth algebraic curve $B$. Let $X \subseteq V$ be an open embedding into a
smooth projective surface $V$ such that $f$ extends to a $\PP^1$-fibration
$\varphi \colon V \to \bar{B}$ over a smooth projective curve $\bar{B}$, where
$B \subseteq \bar{B}$ is the smooth completion.

Since $\rho(X)=0$, $\Pic(V)$ is finitely generated. Consequently, using the Hodge theory and the long exact cohomology sequence induced from the exponential sequence of analytic sheaves, it follows that $b_1(V)=0$. As a general fiber of $\varphi$ is isomorphic to $\PP^1$, which is connected, the induced homomorphism
\[
\varphi_\ast \colon H_1(V;\Z) \to H_1(\bar{B};\Z)
\]
is surjective. Hence $b_1(\bar{B}) \le b_1(V)=0$, which implies that
$\bar{B} \cong \PP^1$ and thus $B$ is rational.

Since $\bar{B} \cong \PP^1$, one general fiber of $\varphi$, a cross-section of
$\varphi$, and the compactifications (in $V$) of all but one irreducible component of each singular fiber of $f$ freely generate $\Pic(V)$.

If $B=\PP^1$, it follows that one general fiber of $f$ and all but one irreducible
component of each singular fiber of $f$ freely generate $\Pic(X)$. Thus, $\rho(X)=\rank_{\Z}\Pic(X)\ge 1$, which
contradicts the assumption that $\rho(X)=0$. Hence $B$ must be affine, proving~(1).

Now let $B$ be a smooth affine rational curve. From the above description of the free
generators of $\Pic(V)$, it follows that all but one irreducible component of each
singular fiber of $f$ freely generate $\Pic(X)$. Therefore, $\rho(X)=\rank_{\Z}\Pic(X)\ge 1$ whenever $f$ has a reducible fiber. This would again contradict the assumption that $\rho(X)=0$. Hence, all fibers of $f$ are irreducible, proving~(2).
\end{proof}

Now we will quote two useful results from~\cite{JSXZ2024}.

\begin{lem}[{cf.~\cite[Lemma~2.6, Corollary~2.7]{JSXZ2024}}]
Let $X$ be a smooth quasi-projective surface and $\pi\colon X\to B$ a $\G_m$-bundle over a smooth algebraic curve $B$. Then the following holds.
\begin{enumerate}[\indent\rm(1)]
\item There exists a finite \'etale cover $B'\to B$ of degree at most~$2$ inducing a
finite \'etale cover $X':=X\times_B B'\to X$ such that $X'\to B'$ is an untwisted
$\G_m$-bundle.
\item The logarithmic Kodaira dimensions satisfy
\[
\bar\kappa(X')=\bar\kappa(X)=\bar\kappa(B)=\bar\kappa(B').
\]
\item If $\pi$ is untwisted and $B$ is a smooth affine rational curve, then $\pi$ is
a trivial $\G_m$-bundle.
\item If $X$ is a smooth affine surface and $\pi$ is a $\G_m$-bundle, then $B$ is
affine.
\end{enumerate}
\end{lem}

\begin{lem}[{cf.~\cite[Lemma~2.14]{JSXZ2024}}]
\label{Lem: A^1-bundle over affine rational curve is trivial}
Any $\A^1$-bundle over a smooth affine rational curve is trivial.
\end{lem}

Now we list certain results that are used repeatedly in the context of proper descent of algebraic varieties.

\begin{lem}\label{Lem: Finite morphism basic properties}
Let $f\colon V\to W$ be a finite surjective morphism between smooth, irreducible
complex algebraic varieties. Then the following holds.
\begin{enumerate}[\indent\rm(1)]
\item $\bar\kappa(V)\ge \bar\kappa(W)$. Moreover, if $f$ is \'etale, then
$\bar\kappa(V)=\bar\kappa(W)$ \textup{(cf.~\cite{Iit1982})}.
\item The induced homomorphism
\[
f_{\ast}: H_i(V;\Q)\to H_i(W;\Q)
\]
is surjective for all $i$. In particular, $b_i(V)\ge b_i(W)$ for all $i$
\textup{(cf.~\cite{Gie1964})}.
\item There is an induced injective homomorphism (by the projection formula)
\[
f^{\ast}\otimes \mathrm{Id}_{\Q}\colon
\Pic(W)\otimes_{\Z}\Q\to \Pic(V)\otimes_{\Z}\Q.
\]
In particular, $\rho(W)\le \rho(V)$.\\
\end{enumerate}
\end{lem}

\subsection{$3$-manifold at infinity and $\pi_1^\infty$ of normal affine surfaces}\label{subsec:fundamental-group-at-infinity}

Let $X$ be a normal affine surface over $\C$. By resolution of singularities, there exists a smooth projective surface $V$ containing $X$ as a dense Zariski open subset such that $D:=V\setminus X$ is a connected simple normal crossing divisor (SNC divisor, for short). The pair $(V,D)$ is called a \emph{log-smooth completion} of $X$.

Regard $V$ as a complex analytic manifold equipped with its Euclidean topology. For a sufficiently small positive real number $\varepsilon$, let $N_\varepsilon(D)$ denote the tubular neighbourhood of $D$ consisting of all points whose distance from $D$ is at most $\varepsilon$. Then $D$ is a strong deformation retract of $N_\varepsilon(D)$, and the boundary $M_X:=\partial N_\varepsilon(D)$ is a compact connected orientable $\mathscr{C}^{\infty}$-smooth $3$-manifold. Moreover, $M_X$ is a strong deformation retract of the punctured tubular neighbourhood $N_\varepsilon(D)\setminus D$. The manifold $M_X$ is called the \emph{$3$-manifold at infinity} of $X$.

Mumford \cite{Mum1961} showed that $M_X$ is naturally realized as the plumbing manifold associated with the weighted dual graph of the boundary divisor $D$. In particular, $M_X$ is a \emph{graph manifold} in the sense of Waldhausen \cite{Wal1967}. We refer the reader to Eisenbud and Neumann \cite{EN1985} for a detailed treatment of graph manifolds, their plumbing descriptions, and their relation to the JSJ decomposition.

Recall that a connected $3$-manifold $M$ is called \emph{prime} if, whenever $M=M_1\#M_2$, one of the summands $M_1$ or $M_2$ is homeomorphic to the $3$-sphere $S^3$. It is called \emph{irreducible} if every smoothly embedded $2$-sphere in $M$ bounds a $3$-ball. By a theorem of W.~D.~Neumann \cite[Theorem~5.1]{Neu1989}, the $3$-manifold at infinity $M_X$ is prime. Furthermore, every connected orientable prime $3$-manifold is either irreducible or homeomorphic to $S^2\times S^1$ (see \cite[Lemma~3.13]{Hem1976}).

The \emph{fundamental group at infinity} of $X$ is defined by
$\pi_1^\infty(X):=\pi_1(M_X)$, or equivalently, as the fundamental group of the boundary of a sufficiently small tubular neighbourhood of the boundary divisor $D$. This group is independent of the choice of $\varepsilon$, provided that $\varepsilon$ is sufficiently small, and also independent of the chosen log-smooth completion (see, for example, \cite[before Theorem~1.2.11]{GMM2021}).

The fundamental group at infinity admits an intrinsic topological description in terms of complements of compact subsets of $X$. Indeed, let $K_1\subseteq K_2\subseteq\cdots$ be an exhaustion of $X$ by compact subsets such that $X=\bigcup_{i=1}^{\infty}K_i$. Then, for every $i\leq j$, the inclusion
$\eta_{ij}\colon X\setminus K_j\hookrightarrow X\setminus K_i$
induces a homomorphism
$(\eta_{ij})_\ast\colon\pi_1(X\setminus K_j)\to\pi_1(X\setminus K_i)$.
Moreover, $\eta_{ii}=\id_{X\setminus K_i}$ for each $i$, and $\eta_{ik}=\eta_{ij}\circ\eta_{jk}$ for all $i\le j \le k$. Thus, $\{(\pi_1(X\setminus K_i))_{i\in \N},((\eta_{ij})_\ast)_{i\le j \in \N}\}$ forms an inverse system of groups and transition morphisms over the poset $(\N, \subseteq)$, and
\[
\pi_1^\infty(X)\cong\varprojlim_i\pi_1(X\setminus K_i).
\]
This inverse limit is independent of the chosen exhaustion of $X$.

\subsubsection{Mumford's presentation for rational tree boundary divisors}

We recall a presentation of the fundamental group at infinity due to Mumford \cite[\S 1.2]{Mum1961} (see also \cite[\S 1.3.9]{GMM2021}). More generally, Mumford described the fundamental group of the boundary of a tubular neighbourhood of a compact, connected divisor with simple normal crossings on a $2$-dimensional complex manifold, provided that its dual graph has no non-trivial loops. The presentation depends on the intersection matrix of the divisor and the genera of its irreducible components. Since we shall only need the case where every irreducible component of the divisor is a smooth rational curve, we record only this special case.

Let $X$ be a normal affine surface, and let $(V, D)$ be a log-smooth completion of $X$. Assume that the connected boundary divisor $D=D_1+\cdots+D_n$ has smooth rational irreducible components and that its dual graph $\Gamma=\Gamma(D)$ is a tree. Let $s_{ij}:=(D_i\cdot D_j)$ denote the entries of the intersection matrix of $D$. Note that $s_{ij}=0 \text{ or } 1$ if $i\neq j$ and $s_{ii}<0$ for all $1\le i,j\le n$. For each irreducible component $D_i$, let $e_i$ denote the corresponding meridian generator of $\pi_1^\infty(X)$. If $\varepsilon_i$ is the vertex of $\Gamma$, corresponding
to $D_i$ then we shall say that $e_i$ is the generator of $\pi_1^\infty(X)$ corresponding to $\varepsilon_i$.

\begin{prop}[Mumford, cf.~\cite{Mum1961}]\label{prop: Mumford's presentation for tree}
With the above notation,
\[
\pi_1^\infty(X)=
\Bigl\langle
e_1,\ldots,e_n
\ \Big|\
\prod_{j=1}^{n}e_j^{\,s_{ij}}=1
\ \text{for each }i=1,\ldots,n,
\ \text{and }\
[e_i,e_j]=1
\ \text{whenever }D_i\cap D_j\neq\varnothing
\Bigr\rangle.
\]
\end{prop}

\subsubsection{Wagreich's presentation for rational unicyclic boundary divisors}

We now recall a generalization of Mumford's presentation due to Wagreich
\cite[Proposition~2.2]{Wag1971}. More generally, Wagreich described the
fundamental group at infinity of a smooth affine surface whose boundary divisor
has a dual graph containing a unique cycle. Since this is the only case needed
later, we record the corresponding presentation.

Let $X$ be a smooth affine surface, and let $(V,D)$ be a log-smooth completion
of $X$. Assume that the dual graph $\Gamma=\Gamma(D)$ has a unique cycle and no
centres. Choose a vertex $\varepsilon_0$ on the unique cycle of $\Gamma$.
Let $\Gamma_1,\ldots,\Gamma_s$ be the connected components of
$\Gamma-\{\varepsilon_0\}$, and assume that $\Gamma_1$ contains the unique
cycle. Let $\varepsilon_1,\ldots,\varepsilon_r$ be the vertices of $\Gamma_1$ and let $\varepsilon_{r+1},\ldots,\varepsilon_t$ denote the remaining vertices of $\Gamma$. For each vertex $\varepsilon_i$, let $D_i$ be the corresponding irreducible component of $D$, let $e_i$ denote the corresponding meridian generator of $\pi_1^\infty(X)$, and write $s_{ij}:=(D_i\cdot D_j)$ for all $0\le i,j\le t$. 

\begin{prop}[Wagreich, {cf. \cite[Proposition~2.2]{Wag1971}}]
\label{prop:Wagreich-cycle-presentation}
With the above notation, the fundamental group at infinity $\pi_1^\infty(X)$ is a free group on the generators $e_0,e_1,\ldots,e_t,u$
subject to the relations
\[
B_0:\quad
e_0^{s_{0,0}}
e_1^{s_{0,1}}
\cdots
e_{r-1}^{s_{0,r-1}}
u^{-1}e_ru
e_{r+1}^{s_{0,r+1}}
\cdots
e_t^{s_{0,t}}
=1,
\]
\[
B_i:\quad
e_0^{s_{i,0}}
\cdots
e_t^{s_{i,t}}
=1,
\qquad i\neq 1,r,
\]
\[
B_r:\quad
ue_0u^{-1}
e_1^{s_{r,1}}
\cdots
e_t^{s_{r,t}}
=1,
\]
\[
A_{i,j}:\quad
e_ie_j=e_je_i,
\quad \text{whenever } \varepsilon_i \text{ is joined to } \varepsilon_j; \ (i,j)\neq(0,r) \text{ or } (r,0),
\]
\[
A_{0,r}:e_0u^{-1}e_ru=u^{-1}e_rue_0.
\]
\end{prop}

\subsubsection{A useful application} The presentations of Mumford and Wagreich provide an effective method for
computing the fundamental group at infinity directly from a log-smooth
completion. We illustrate this by computing
$\pi_1^\infty(\C^*\times\C^*)$.

\begin{prop}\label{Prop: Fundamental group at infinity of 2-torus}
The fundamental group at infinity of $\C^*\times \C^*$ is isomorphic to $\Z^3$.
\end{prop}

\begin{proof}
Let $X=\C^*\times\C^*$. Consider the standard log-smooth completion $X\hookrightarrow\PP^1\times\PP^1$. Its boundary divisor $D=D_0+D_1+D_2+D_3$ consists of four smooth rational curves forming a cycle. The dual graph $\Gamma=\Gamma(D)$ is the $4$-cycle $C_4$, and each component satisfies $(D_i^2)=0$. Clearly $\pi_1(\Gamma)\cong \Z$, and $\Gamma$ has no centers. Let $\varepsilon_i$ denote the vertex of $\Gamma$ corresponding to $D_i$, and let $e_i$ be the associated generator in the above presentation due to Wagreich for all $0\le i \le 3$. Choose $\varepsilon_0$ as the distinguished vertex in the cycle. Thus the weighted dual graph of $D$ is:
	$$
	
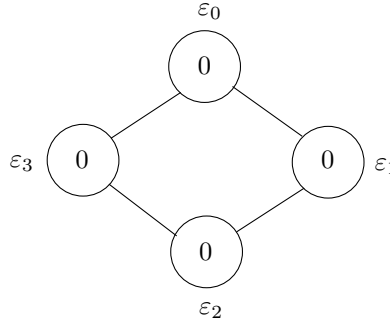
\begin{tikzpicture}[x=0.75pt,y=0.75pt,yscale=-1,xscale=1]
	
	\draw   (302,51) .. controls (302,41.06) and (310.06,33) .. (320,33) .. controls (329.94,33) and (338,41.06) .. (338,51) .. controls (338,60.94) and (329.94,69) .. (320,69) .. controls (310.06,69) and (302,60.94) .. (302,51) -- cycle ;
	
	\draw   (241,98) .. controls (241,88.06) and (249.06,80) .. (259,80) .. controls (268.94,80) and (277,88.06) .. (277,98) .. controls (277,107.94) and (268.94,116) .. (259,116) .. controls (249.06,116) and (241,107.94) .. (241,98) -- cycle ;
	
	\draw   (364,99) .. controls (364,89.06) and (372.06,81) .. (382,81) .. controls (391.94,81) and (400,89.06) .. (400,99) .. controls (400,108.94) and (391.94,117) .. (382,117) .. controls (372.06,117) and (364,108.94) .. (364,99) -- cycle ;
	
	\draw   (303,144) .. controls (303,134.06) and (311.06,126) .. (321,126) .. controls (330.94,126) and (339,134.06) .. (339,144) .. controls (339,153.94) and (330.94,162) .. (321,162) .. controls (311.06,162) and (303,153.94) .. (303,144) -- cycle ;
	
	\draw    (308,64) -- (273,87) ;
	
	\draw    (370,112) -- (336,134) ;
	
	\draw    (369,86) -- (334,61) ;
	
	\draw    (306,136) -- (272,110) ;
	
	\draw (315.3,45) node [anchor=north west][inner sep=0.75pt]   [align=left] {0};
	
	\draw (377.2,92.2) node [anchor=north west][inner sep=0.75pt]   [align=left] {$0$};
	
	\draw (316,138) node [anchor=north west][inner sep=0.75pt]   [align=left] {$0$};
	
	\draw (254,92) node [anchor=north west][inner sep=0.75pt]   [align=left] {$0$};
	
	\draw (315.3,18) node [anchor=north west][inner sep=0.75pt]   [align=left] {$\varepsilon_0$};
	
	\draw (404.2,96) node [anchor=north west][inner sep=0.75pt]   [align=left] {$\varepsilon_1$};
	
	\draw (221,94.5) node [anchor=north west][inner sep=0.75pt]   [align=left] {$\varepsilon_3$};
	
	\draw (316,168) node [anchor=north west][inner sep=0.75pt]   [align=left] {$\varepsilon_2$};
	\end{tikzpicture}
	$$
    Applying Proposition~\ref{prop:Wagreich-cycle-presentation}, we obtain that $\pi_1^\infty(X)$ is generated by $e_0,e_1,e_2,e_3,u$ subject to the relations
    $$
	\begin{array}{ccc}
	(0) \cdots & e_1u^{-1}e_3u=1; & [e_0,e_1]=1;\\
	(1) \cdots & e_0e_2=1; & [e_1,e_2]=1;\\
	(2) \cdots & e_1e_3=1; & [e_2,e_3]=1;\\
	(3) \cdots & ue_0u^{-1}e_2=1; & [e_0,u^{-1}e_3u]=1;
	\end{array}
	$$
    From the relations (1) and (2), we obtain $e_2=e_0^{-1}$ and $e_3=e_1^{-1}$. Substituting these into the relations (0) and (3) gives
\[
e_1u^{-1}e_1^{-1}u=1,
\qquad
ue_0u^{-1}e_0^{-1}=1,
\]
that is, $[e_1,u]=1$ and $[e_0,u]=1$. Together with $[e_0,e_1]=1$, it follows that $e_0$, $e_1$, and $u$ commute pairwise. Since $e_2=e_0^{-1}$ and $e_3=e_1^{-1}$, every generator is expressed in terms
of $e_0,e_1$, and $u$. The remaining relations
\[
[e_1,e_2]=1,\qquad
[e_2,e_3]=1,\qquad
[e_0,u^{-1}e_3u]=1
\]
are then automatically satisfied. Therefore
\[
\pi_1^\infty(X)= \Bigl\langle e_0, e_1,u
\ :
[e_0,e_1]=[e_0, u]=[e_1,u]=1
\Bigr\rangle
\cong 
\Z e_0\oplus \Z e_1\oplus \Z u
\cong \Z^3.
\]
This completes the proof.
\end{proof}

\subsection{The class $\mathscr{S}_0$ of smooth affine surfaces}\label{Subsec_Class S_0}
Following Pe{\l}ka--Ra{\'z}ny \cite{PR2021}, denote by $\mathscr{S}_0$ the class of smooth affine
surfaces $S$ satisfying the following three properties:
\begin{enumerate}[\indent\rm(1)]
\item the coordinate ring $\C[S]$ is UFD;
\item $\C[S]^{\ast}=\text{the units in } \C[S]=\C^{\ast}$;
\item $\bar\kappa(S)=0$.
\end{enumerate}

The class $\mathscr{S}_0$ admits an explicit algebraic description, as given by the following theorem.

\begin{thm}[{cf. \cite[Theorem 1.1]{PR2021}}]\label{Thm: PR classification S_0}
A smooth affine surface $S\in\mathscr{S}_0$ if and only if there exist monic polynomials
$p_1,p_2\in\C[t]$ such that $S$ is isomorphic to
\[
S_{p_1,p_2}
:=\Spec \C[x_1,x_2]\Bigl[
\bigl(x_2x_1^{-1}-p_1(x_1^{-1})\bigr)x_1^{-1},
\,
\bigl(x_1x_2^{-1}-p_2(x_2^{-1})\bigr)x_2^{-1}
\Bigr].
\]

Moreover, if $p_1,p_2,p_1',p_2'\in\C[t]$ are monic, then
\[
S_{p_1,p_2}\cong S_{p_1',p_2'}
\quad\text{if and only if}\quad
\{p_1,p_2\}=\{p_1',p_2'\}.
\]
In particular, distinct pairs $\{p_1,p_2\}$ yield pairwise
non-isomorphic surfaces in\/ $\mathscr{S}_0$.
\end{thm}

In \cite{PR2021}, the authors computed explicitely the fundamental group at infinity of $S$ as follows.

\begin{prop}[{cf. \cite[Proposition~4.5]{PR2021}}]
\label{Prop: Fundamental group at infinity of surfaces in S_0 by P-R}
The fundamental group at infinity of $S$ admits the presentation
\begin{equation}\label{eq:pi1infinity}
\pi_1^\infty(S)
=
\left\langle
\delta_1,\delta_2,\lambda
\;\middle|\;
\delta_1=[\gamma_2,\lambda^{-1}],
\;
\delta_2=[\gamma_1,\lambda],
\;
[\gamma_1,\gamma_2]=1
\right\rangle,
\end{equation}
where $\gamma_j=\delta_j^{\,d_j}$.

Moreover, for $j\in\{1,2\}$, the loop $\gamma_j$ is the attaching circle of the $(-d_j)$-framed $2$-handle in \cite[Figure~3]{PR2021}, and $\lambda$ is the belt sphere of the $0$-framed $2$-handle.
\end{prop}

\section{\bf Proof of the Main Theorem}\label{se3}

Let $Y$ be a smooth affine surface properly dominated by the complex algebraic $2$-torus
$X=\C^*\times\C^*$.
Our goal is to determine all possible isomorphism classes of $Y$.

A recurring feature of our arguments is that many geometric and topological
properties are preserved under finite \'etale coverings. Consequently, it is
convenient to first replace $Y$, whenever necessary, by a suitable finite
\'etale cover so that the induced homomorphism on fundamental groups becomes
surjective. This reduction will simplify the subsequent arguments and will be
invoked repeatedly throughout the proofs.

We begin by recording this elementary but useful observation.

\begin{rem}[Reduction to the surjective fundamental group case]
\label{Rem: pi_1 level map is surjective}

Let $f\colon X\to Y$ be a finite surjective morphism between smooth affine
varieties. We first reduce to the case where the induced homomorphism
$f_\ast\colon \pi_1(X)\to\pi_1(Y)$ is surjective.

Suppose that $f_\ast$ is not surjective. Since the image
$f_\ast(\pi_1(X))$ has finite index in $\pi_1(Y)$ by a theorem of Serre
(cf.~\cite{Ser1959}), covering space theory yields a connected finite
\'etale covering $g\colon Z\to Y$ corresponding to the subgroup
$f_\ast(\pi_1(X))\subseteq\pi_1(Y)$. Consequently, $f$ admits the
following factorization:
\[
\begin{tikzcd}[column sep=3em,row sep=2em]
&
Z \arrow[dr,"g"] &
\\
X \arrow[ur,"h"] \arrow[rr,"f"'] &&
Y
\\[-5em]
&
{\scalebox{2}{$\circlearrowright$}}
\end{tikzcd}
\]
where $h\colon X\to Z$ is again a finite surjective morphism satisfying
$f=g\circ h$. Moreover,
$g_\ast\bigl(\pi_1(Z)\bigr)=f_\ast\bigl(\pi_1(X)\bigr)$, and hence
$h_\ast\colon\pi_1(X)\to\pi_1(Z)$ is surjective. Furthermore,
$\deg(f)=\deg(g)\deg(h)$.

Therefore, after replacing $Y$ by the finite \'etale covering $Z$, we may
always assume that the induced homomorphism on fundamental groups is
surjective. This reduction will be invoked repeatedly throughout the
proofs in this section.
\end{rem}

The above reduction enables us to replace the target surface by a suitable
finite \'etale covering whenever necessary, without changing its smoothness
or logarithmic Kodaira dimension. The following result describes the
connected finite \'etale coverings of the complement of a finite subset of
$\G_m^2$. It will play a fundamental role throughout the remainder of this
section, since it guarantees that every connected finite \'etale cover of
$\G_m^2$ is again isomorphic to $\G_m^2$.

\begin{prop}\label{Prop: Finite covering of complement of a finite subset from 2-torus}
Let $S=\{p_1,\ldots,p_k\}\subseteq\G_m^2$ be a finite subset, where $k\ge0$, and
let $S=\varnothing$ when $k=0$. Suppose that
\[
\varphi\colon V\to\G_m^2\setminus S
\]
is a connected finite \'etale morphism. Then there exists a finite \'etale
endomorphism
\[
\Phi\colon\G_m^2\to\G_m^2
\]
such that $V\cong\G_m^2\setminus\Phi^{-1}(S)$. In particular, when
$S=\varnothing$, we have $V\cong\G_m^2$.
\end{prop}

\begin{proof}
Let $W$ be the normalization of $\G_m^2$ in the function field $\C(V)$ of
$V$. Then $W$ is a normal affine surface containing $V$ as a Zariski open
subset, and the finite morphism $\varphi$ extends uniquely to a finite
surjective morphism
\[
\overline{\varphi}\colon W\to\G_m^2.
\]

Since $\varphi$ is \'etale, the morphism $\overline{\varphi}$ is unramified
over $\G_m^2\setminus S$. Hence every irreducible component of the branch
locus of $\overline{\varphi}$ is contained in $S$. As $S$ has codimension $2$
in $\G_m^2$, the Zariski--Nagata purity theorem implies that the branch locus of $\overline{\varphi}$ is empty. Therefore $\overline{\varphi}$ is unramified. Since $\G_m^2$ is
smooth, $W$ is a normal algebraic surface (hence Cohen--Macaulay), and $\overline{\varphi}$ is finite (hence
equidimensional), Serre's miracle flatness implies that
$\overline{\varphi}$ is flat. As it is already unramified, it follows that
$\overline{\varphi}$ is \'etale. Since every \'etale morphism is smooth, it
follows that $W$ is smooth.

Passing to the associated complex analytic spaces, we obtain a connected
finite-sheeted analytic covering
\[
\overline{\varphi}^{\rm an}\colon
W^{\rm an}\to(\G_m^2)^{\textrm{an}}=\C^*\times \C^*.
\]
Since $\pi_1(\C^*\times \C^*)\cong\Z^2$, the covering
$\overline{\varphi}^{\rm an}$ corresponds to the finite-index subgroup
\[
H=\overline{\varphi}^{\rm an}_*\bigl(\pi_1(W^{\rm an})\bigr)\subseteq\Z^2.
\]

As $H$ is a free abelian group of rank $2$, choose a $\Z$-basis of $H$. If
the corresponding basis vectors are the columns of the integer matrix
\[
T=
\begin{pmatrix}
a&b\\
c&d
\end{pmatrix},
\qquad
\det(T)\neq0,
\]
then $H=T(\Z^2)$. Consider the associated monomial endomorphism
\[
\Phi\colon\G_m^2\to\G_m^2,\qquad
(z,w)\longmapsto(z^aw^b,z^cw^d).
\]
It is well known that $\Phi$ is a connected finite \'etale endomorphism
satisfying
\[
\Phi^{\rm an}_*\bigl(\pi_1(\C^*\times \C^*)\bigr)=H.
\]

Hence, by the classification of connected covering spaces,
$\overline{\varphi}^{\rm an}$ and $\Phi^{\rm an}$ are isomorphic as analytic
coverings of $\C^*\times \C^*$. The Generalized Riemann Existence Theorem
then implies that $\overline{\varphi}$ and $\Phi$ are isomorphic as finite
\'etale covers of $\G_m^2$. Thus, after identifying $W$ with $\G_m^2$, we may
assume that $\overline{\varphi}=\Phi$.

Finally, since $\varphi$ is the restriction of $\overline{\varphi}$ to
$\G_m^2\setminus S$, we obtain
\[
V\cong
W\setminus\overline{\varphi}^{-1}(S)
\cong
\G_m^2\setminus\Phi^{-1}(S).
\]
Since $\Phi$ is finite, the set $\Phi^{-1}(S)$ is finite, completing the
proof.
\end{proof}

We now proceed to classify smooth affine surfaces properly dominated by
$\C^*\times\C^*$. Since
$\bar\kappa(\C^*\times\C^*)=0$ and logarithmic Kodaira dimension does not
increase under finite surjective morphisms by
Lemma~\ref{Lem: Finite morphism basic properties}(1), we necessarily have
$\bar\kappa(Y)\le0$. Thus it suffices to consider the two cases
$\bar\kappa(Y)=-\infty$ and $\bar\kappa(Y)=0$, which we treat separately.

\subsection{The case $\bar\kappa(Y)=-\infty$}

\begin{thm}\label{Thm: Affine surfaces with kappa negative properly dominated by 2-torus}
Let $Y$ be a smooth affine surface with $\bar\kappa(Y)=-\infty$. If $Y$ is properly
dominated by $\C^*\times\C^*$, then $Y$ is isomorphic to either $\A^2$ or
$\A^1\times\C^*$.
\end{thm}

\begin{proof}
Let $X:=\C^*\times\C^*$ and let $f\colon X\to Y$ be a finite surjective
morphism. If $\deg(f)=1$, then $f$ is a finite birational morphism and hence an
isomorphism by Zariski's Main Theorem, which cannot occur since $\bar\kappa(X)=0$ whereas $\bar\kappa(Y)=-\infty$. Thus $\deg(f)>1$.

By Remark~\ref{Rem: pi_1 level map is surjective}, we may factorize $f$ as $f=g\circ h$, where
$g\colon Z\to Y$ is a finite \'etale morphism and
$h\colon X\to Z$ is a finite surjective morphism inducing a surjection
$h_\ast\colon\pi_1(X)\to\pi_1(Z)$.
Since logarithmic Kodaira dimension is preserved under finite \'etale
coverings by Lemma \ref{Lem: Finite morphism basic properties}(1), we have $\bar\kappa(Z)=\bar\kappa(Y)=-\infty$.

By the results of Fujita--Miyanishi--Sugie
\cite{Fuj1979,MS1980,Sug1980},
the surface $Z$ admits an $\A^1$-fibration
$\varphi\colon Z\to B$
over a smooth algebraic curve $B$.
Furthermore, $\rho(Z)=0$ by
Lemma~\ref{Lem: Finite morphism basic properties}(3), since $X$ is
factorial. Hence
Lemma~\ref{Lem: Q-factorial fibration}
implies that $B$ is affine and that
$\varphi$ has no reducible fibers.

Since $\pi_1(X)\cong\Z^2$ and
$h_\ast\colon\pi_1(X)\to\pi_1(Z)$ is surjective,
the group $\pi_1(Z)$ is abelian.
Moreover, the induced homomorphism
$\varphi_\ast\colon\pi_1(Z)\to\pi_1(B)$ is surjective, and therefore
$\pi_1(B)$ is also abelian. Hence $B\cong\A^1$ or $\C^*$.

It remains to show that $\varphi$ has no multiple fibers.
Suppose, to the contrary, that $\varphi$ has a multiple fiber.
By the ramified covering trick (cf. \cite[Lemma 1.1.9]{GMM2021}), there exists a finite morphism
$B'\to B$ inducing a finite \'etale covering
$p\colon Z'\to Z$, where
$Z':=\overline{Z\times_BB'}$ is the normalization of the fiber product,
such that the induced $\A^1$-fibration
$\varphi'\colon Z'\to B'$
has no multiple fibers.
In the process of eliminating multiple fibers, the fibration $\varphi'$ necessarily acquires a reducible fiber.
Hence
$\rho(Z')>0$
by Lemma~\ref{Lem: Q-factorial fibration}.

Now let
$q\colon V\to X$
be the pullback of $p$ via $h$.
Since $h_\ast$ is surjective, $V$ is connected, and the induced morphism
$h'\colon V\to Z'$ is finite surjective. Moreover, we have the commutative
diagram
\[
\begin{tikzcd}[column sep=5em,row sep=2em]
V \arrow[r,"h'"] \arrow[d,"q"'] &
Z' \arrow[d,"p"]\\
X \arrow[r,"h"'] &
Z
\arrow[phantom,from=1-1,to=2-2,
"{\raisebox{-0.15em}{\scalebox{2.2}{$\circlearrowleft$}}}" description]
\end{tikzcd}
\]

Since $q$ is a finite covering,
Proposition~\ref{Prop: Finite covering of complement of a finite subset from 2-torus}
shows that $V\cong\C^*\times\C^*$.
Hence $Z'$ is again properly dominated by
$\C^*\times\C^*$.
Repeating the above argument with $Z'$ in place of $Z$, it turns out that $\rho(Z')=0$, contradicting the existence of a reducible fiber of $\varphi'$.
Therefore $\varphi$ has no multiple fibers.
Since $\varphi$ also has no reducible fibers,
it follows that $\varphi$ is an $\A^1$-bundle.
As $B$ is an affine rational curve,
Lemma~\ref{Lem: A^1-bundle over affine rational curve is trivial}
implies that $\varphi$ is trivial.

We now distinguish two cases.

\smallskip
\noindent\textbf{Case 1.} \emph{$B\cong\A^1$.}

Then $Z\cong\A^2$.
Since $g$ is a finite topological covering and
$\pi_1(Z)=1$, the group $\pi_1(Y)$ is finite.
On the other hand,
$\A^2$ is contractible, so $Y$ is a $K(\pi,1)$-space. Since $Y$ is an affine surface, it is a finite-dimensional CW complex; therefore $\pi_1(Y)$ is torsion-free using \cite[Proposition 2.45]{Hat2002}, forcing $\pi_1(Y)$ to be trivial.
Consequently, $g$ is an isomorphism, and hence
$Y\cong Z\cong\A^2$.\footnote{\textsc{Aliter.}
Since $Z\cong\A^2$ and $Y$ is the image of $Z$ under a finite morphism of
smooth affine surfaces, it follows from
\cite{Miy1980b,GS1984}
that $Y\cong\A^2$.}

\smallskip
\noindent\textbf{Case 2.} \emph{$B\cong\C^*$.}

Then $Z\cong\A^1\times\C^*$,
which is an Eilenberg--MacLane $K(\pi,1)$-space.
Since $g\colon Z\to Y$ is a finite topological covering,
$Y$ is also an Eilenberg--MacLane $K(\pi,1)$-space.
By \cite[Theorem~5.3]{GGH2023},
$Y$ is an $\A^1$-bundle over a smooth algebraic curve $C$. Again, $\rho(Y)=0$ by
Lemma~\ref{Lem: Finite morphism basic properties}(3),
since $X$ is factorial.
Hence
Lemma~\ref{Lem: Q-factorial fibration}
implies that $C$ is affine.
Moreover,
$e(Z)=e(\A^1\times\C^*)=0$, so
$e(Y)=0$, and therefore
$e(C)=0$.
Thus $C\cong\C^*$, and
Lemma~\ref{Lem: A^1-bundle over affine rational curve is trivial}
shows that
$Y\cong\A^1\times\C^*$.\footnote{\textsc{Aliter.}
Since $Z\cong\A^1\times\C^*$,
one may directly apply
\cite{Fur1989}
to conclude that
$Y\cong\A^1\times\C^*$,
as $Y$ is non-contractible because
$g\colon Z\to Y$
is a finite covering.}
\end{proof}

\subsection{The case $\bar\kappa(Y)=0$}

Having completed the classification in the case
$\bar\kappa(Y)=-\infty$, we now turn to the remaining case
$\bar\kappa(Y)=0$. The argument is considerably more subtle. Rather than
proceeding directly to the classification, we first establish a sequence of
topological and geometric obstructions that progressively eliminate possible
fundamental groups and geometric configurations of the target surface $Y$. These
auxiliary results ultimately lead to the classification theorem for smooth
affine surfaces of logarithmic Kodaira dimension zero. The argument proceeds
in several steps:

\begin{enumerate}[\indent\rm(1)]
\item We first prove that the fundamental group of $Y$ cannot be finite. A
key ingredient is to show that no smooth affine surface in the class
$\mathscr{S}_0$ can be properly dominated by $\C^*\times\C^*$.

\item Next, we exclude smooth affine surfaces with an abelian fundamental
group of rank $1$ and positive Euler characteristic.

\item Finally, combining these obstructions with Kojima's classification of
strongly minimal affine surfaces of logarithmic Kodaira dimension zero, we
obtain the desired classification.
\end{enumerate}

\begin{prop}\label{Prop: Affine surface properly dominated by 2-torus and with finite pi_1}
Let $Y$ be a smooth affine surface with $\bar\kappa(Y)=0$. If $Y$ is properly
dominated by $\C^*\times\C^*$, then $\pi_1(Y)$ cannot be finite.
\end{prop}

\begin{proof}
Let $X:=\C^*\times\C^*$ and let $f\colon X\to Y$ be a finite surjective
morphism. If $\deg(f)=1$, then $f$ is a finite birational morphism between
smooth affine surfaces. Hence, by Zariski's Main Theorem, $f$ is an
isomorphism. This is impossible since
$\pi_1(X)\cong\Z^2$, whereas $\pi_1(Y)$ is finite.
Therefore, $\deg(f)>1$.

Using the reduction in
Remark~\ref{Rem: pi_1 level map is surjective},
we factorize $f$ as $f=g\circ h$, where
$g\colon Z\to Y$ is a finite \'etale morphism and
$h\colon X\to Z$ is a finite surjective morphism satisfying
$h_\ast\colon\pi_1(X)\to\pi_1(Z)$ is surjective.

Since $\pi_1(X)\cong\Z^2$, the group $\pi_1(Z)$ is a finitely generated
abelian group. On the other hand, $g$ is a finite covering, so
$g_\ast(\pi_1(Z))$ is a finite-index subgroup of the finite group
$\pi_1(Y)$. Consequently, $\pi_1(Z)$ is finite.

Let $p\colon\widetilde{Z}\to Z$ be the universal covering. Since
$\pi_1(Z)$ is finite, the covering $p$ is finite. Form the fiber product
$\widetilde{X}:=X\times_Z\widetilde{Z}$. Then we obtain the commutative
diagram
\[
\begin{tikzcd}[column sep=5em,row sep=2em]
\widetilde{X}
\arrow[r,"h'"]
\arrow[d,"q"']
&
\widetilde{Z}
\arrow[d,"p"]
\\
X
\arrow[r,"h"']
&
Z
\arrow[phantom,from=1-1,to=2-2,
"{\raisebox{-0.15em}{\scalebox{2.2}{$\circlearrowleft$}}}" description]
\end{tikzcd}
\]

Since $h_\ast$ is surjective, the pullback
$q\colon\widetilde{X}\to X$ is connected. Moreover, $q$ is finite \'etale,
and hence Proposition~\ref{Prop: Finite covering of complement of a finite subset from 2-torus}
implies that $\widetilde{X}\cong\C^*\times\C^*$.
Therefore, $h'\colon\widetilde{X}\to\widetilde{Z}$ is again a finite
surjective morphism.

Since $\widetilde{X}$ is factorial,
Lemma~\ref{Lem: Finite morphism basic properties}(3) yields
$\rho(\widetilde{Z})=0$, and consequently
$\Cl(\widetilde{Z})$ is finite.
As $\widetilde{Z}$ is simply connected, it follows that
$\widetilde{Z}$ is in fact factorial and
$\Gamma(\widetilde{Z},\mathcal O_{\widetilde{Z}})^*=\C^*$.

Furthermore, since both $p$ and $g$ are finite \'etale morphisms, we have
$\bar\kappa(\widetilde{Z})=\bar\kappa(Z)=\bar\kappa(Y)=0$.
Hence $\widetilde{Z}$ belongs to the class $\mathscr{S}_0$ introduced in \S\ref{Subsec_Class S_0}.

The following proposition shows that no surface in $\mathscr{S}_0$ can be
properly dominated by $\C^*\times\C^*$. Since
$\widetilde{X}\cong\C^*\times\C^*$ properly dominates $\widetilde{Z}$, we
arrive at a contradiction. Hence no such $f$ exists.
\end{proof}

The preceding proposition reduces the problem to determining whether a surface in the class $\mathscr S_0$ can be properly dominated by $\C^*\times\C^*$. The next proposition shows that this never occurs.

We shall use the following theorem, proved independently by Gabai \cite{Gab1992} and Casson--Jungreis \cite{CJ1994}.

\begin{thm}[Seifert Fiber Space Theorem; cf.~\cite{Gab1992,CJ1994}]
\label{SFS Theorem}
Let $M$ be a compact, orientable, irreducible $3$-manifold with infinite fundamental group. Then $M$ is a Seifert fibre space if and only if $\pi_1(M)$ contains an infinite cyclic normal subgroup.
\end{thm}

\begin{prop}
\label{Prop: No affine surface in S_0 can be properly dominated by 2-torus}
No smooth affine surface in the class $\mathscr{S}_0$ is properly dominated by $\C^*\times\C^*$.
\end{prop}

\begin{proof}
Suppose, to the contrary, that there exists a finite surjective morphism
$f\colon X\to Y$, where $X=\C^*\times\C^*$ and $Y\in\mathscr{S}_0$.

Since $f$ is finite, properness induces a homomorphism
$f_\ast^\infty\colon\pi_1^\infty(X)\to\pi_1^\infty(Y)$ by
\cite[Proposition~2.8.8]{GMM2021}. Moreover, by Serre's theorem
(cf.~\cite{Ser1959}), its image has finite index in
$\pi_1^\infty(Y)$. Suppose that $f_\ast^\infty$ is not surjective.
Then, by the factorization argument described in
Remark~\ref{Rem: pi_1 level map is surjective}, there exists a commutative diagram
\[
\begin{tikzcd}[column sep=3em,row sep=2em]
&
S \arrow[dr,"p"] &
\\
X \arrow[ur,"g"] \arrow[rr,"f"'] &&
Y
\\[-5em]
&
{\scalebox{2}{$\circlearrowright$}}
\end{tikzcd}
\]
where $p\colon S\to Y$ is a connected finite \'etale morphism,
$g\colon X\to S$ is a finite surjective morphism,
$f=p\circ g$, and
$g_\ast^\infty\colon\pi_1^\infty(X)\to\pi_1^\infty(S)$
is surjective (see also the proof of
\cite[Lemma~2.8.5]{GMM2021} for details). Since
$Y\in\mathscr{S}_0$, it follows from
\cite[Proposition~4.1]{PR2021} that $Y$ is homotopy equivalent to $S^2$, and hence simply connected. Therefore the connected finite \'etale covering
$p\colon S\to Y$ is trivial, so $p$ is an isomorphism. Consequently,
$f_\ast^\infty=g_\ast^\infty$ is surjective, a contradiction. Thus
$f_\ast^\infty\colon\pi_1^\infty(X)\to\pi_1^\infty(Y)$ is surjective.

Since $\pi_1^\infty(X)\cong \Z^3$ by
Proposition~\ref{Prop: Fundamental group at infinity of 2-torus},
it follows that $\pi_1^\infty(Y)$ is abelian.

\medskip

We first determine the structure of
$\pi_1^\infty(Y)$.

\medskip

\noindent\textbf{Claim.}
\emph{If $Y\in\mathscr{S}_0$ and $\pi_1^\infty(Y)$ is abelian, then
$\pi_1^\infty(Y)\cong\Z$.}

\smallskip

\noindent\emph{Proof of the claim.}
By Proposition~\ref{Prop: Fundamental group at infinity of surfaces in S_0 by P-R},
$\pi_1^\infty(Y)$ admits the presentation
\[
\pi_1^\infty(Y)=
\left\langle
\delta_1,\delta_2,\lambda
\;\middle|\;
\delta_1=[\gamma_2,\lambda^{-1}],
\;
\delta_2=[\gamma_1,\lambda],
\;
[\gamma_1,\gamma_2]=1
\right\rangle,
\]
where $\gamma_i=\delta_i^{\,d_i}$ for $i=1,2$, following the notation of \cite{PR2021}.

Since $\pi_1^\infty(Y)$ is abelian, the generators
$\delta_1$, $\delta_2$, and $\lambda$ commute pairwise. Hence
$[\gamma_1,\lambda]=[\delta_1^{d_1},\lambda]=1$ and
$[\gamma_2,\lambda^{-1}]=[\delta_2^{d_2},\lambda^{-1}]=1$.
Using the defining relations, we obtain
$\delta_1=\delta_2=1$, and consequently
$\gamma_1=\gamma_2=1$. Hence the presentation simplifies to
$\pi_1^\infty(Y)=\langle\lambda\rangle$.
Since there is no relation involving $\lambda$, it follows that
$\pi_1^\infty(Y)\cong\Z$.
This proves the claim.

\medskip

By definition,
$\pi_1(M_Y)=\pi_1^\infty(Y)\cong\Z$, where $M_Y$ denotes the
$3$-manifold at infinity of $Y$. By the discussion in
Section~\ref{subsec:fundamental-group-at-infinity}, $M_Y$ is prime. Hence,
by \cite[Lemma~3.13]{Hem1976}, $M_Y$ is either irreducible or homeomorphic to
$S^2\times S^1$.

If $M_Y$ is irreducible, then $\pi_1(M_Y)\cong\Z$ contains an infinite cyclic
normal subgroup, namely itself. Therefore, by the Seifert Fiber Space Theorem
(Theorem~\ref{SFS Theorem}), $M_Y$ is Seifert fibered. If
$M_Y\cong S^2\times S^1$, then $M_Y$ is also Seifert fibered. Thus, in either
case, $M_Y$ is a Seifert fibered $3$-manifold.

Applying Neumann's plumbing calculus \cite{Neu1981} to the Seifert fibered
manifold $M_Y$, we conclude that the boundary divisor $D$ of a smooth SNC
compactification of $Y$ is irreducible and isomorphic to $\PP^1$.
By Mumford's presentation (Proposition~\ref{prop: Mumford's presentation for tree}),
we have
\[
\pi_1^\infty(Y)\cong
\langle e\mid e^{\,D^2}=1\rangle,
\]
where $e$ denotes the meridian corresponding to the unique irreducible
component $D$ of the boundary divisor. Since
$\pi_1^\infty(Y)\cong\Z$, it follows that $D^2=0$; otherwise,
$\pi_1^\infty(Y)$ would be finite.

On the other hand, since $Y$ is affine, the boundary divisor is supported on an
ample divisor. Hence $D^2>0$, a contradiction.

Therefore no finite surjective morphism
$f\colon\C^*\times\C^*\to Y$ exists. Consequently, no smooth affine surface in
the class $\mathscr{S}_0$ is properly dominated by $\C^*\times\C^*$.
\end{proof}

The next step is to rule out smooth affine surfaces with an abelian
fundamental group of rank $1$ and positive Euler characteristic.

\begin{prop}\label{Prop: No finite map from 2-torus to surface having infinite cyclic pi_1 and positive Euler characteristic}
No smooth affine surface with fundamental group isomorphic to $\Z$ and positive
Euler characteristic can be properly dominated by $\C^*\times\C^*$.
\end{prop}

\begin{proof}
Let $X=\C^*\times\C^*$. Assume, to the contrary, that there exists a finite surjective morphism
$f\colon X\to Y$, where 
$\pi_1(Y)\cong\Z$, and $e(Y)>0$.

If $\deg(f)=1$, then $f$ is a finite birational morphism between smooth affine
surfaces. Hence, by Zariski's Main Theorem, $f$ is an isomorphism. This is
impossible, since $\pi_1(X)\cong\Z^2$ whereas $\pi_1(Y)\cong \Z$. Therefore, $\deg(f)>1$.

Choose such a finite surjective morphism of minimal degree, and denote its degree by $d>1$.

\medskip

\noindent\textbf{Claim.}
\emph{The induced homomorphism
$f_\ast\colon\pi_1(X)\to\pi_1(Y)$ is surjective.}

\smallskip
\noindent\emph{Proof of the claim.}
Suppose that $f_\ast$ is not surjective. By
Remark~\ref{Rem: pi_1 level map is surjective}, the morphism $f$ factors as $g\circ h$ where $g\colon Z \to Y$ is a finite \'etale morphism,
$h\colon X \to Z$ is a finite surjective morphism, and
$h_\ast\colon\pi_1(X)\to\pi_1(Z)$ is surjective. Since $[\pi_1(Y):f_\ast(\pi_1(X))]<\infty$
(cf.~\cite{Ser1959}) and $\pi_1(Y)\cong\Z$, we obtain
$\pi_1(Z)\cong\Z$.
Moreover,
\[
e(Z)=\deg(g)\,e(Y)>0.
\]
As $\deg(h)<d$, this contradicts the minimality of $d$.
Hence the claim follows.

\medskip

Let $p\colon Y'\to Y$ be a connected $2$-fold covering.
Then $Y'$ is again a smooth affine surface with
$\pi_1(Y')\cong\Z$.
Since
$b_0(Y')=1$, $b_1(Y')=1$, and $b_i(Y')=0$ for $i\ge3$,
we have $e(Y')=1-1+b_2(Y')=b_2(Y')$. Furthermore,
\[
e(Y')=\deg(p)\,e(Y)=2e(Y)\ge2,
\]
so that
$b_2(Y')\ge2$.

Since $f_\ast$ is surjective, the covering $p$ pulls back to a connected finite
covering, say $q\colon X' \to X$, via $f$ with the following commutative diagram
\[
\begin{tikzcd}[column sep=5em,row sep=2em]
X' \arrow[r,"q"] \arrow[d,"f'"'] &
X \arrow[d,"f"]
\\
Y' \arrow[r,"p"'] &
Y
\arrow[phantom,from=1-1,to=2-2,
"{\raisebox{-0.15em}{\scalebox{2.2}{$\circlearrowleft$}}}" description]
\end{tikzcd}
\]
The morphism $q\colon X'\to X$ is a finite unramified covering. Since every connected
finite unramified covering of $\C^*\times\C^*$ is again isomorphic to
$\C^*\times\C^*$ (Proposition~\ref{Prop: Finite covering of complement of a finite subset from 2-torus}),
we have
$X'\cong\C^*\times\C^*$.
Consequently, $b_2(X')=1$, whereas $b_2(Y')\ge2$. This contradicts Lemma~\ref{Lem: Finite morphism basic properties}(2) when applied to the finite surjective morphism $f'\colon X'\to Y'$. Hence no smooth affine surface with fundamental group isomorphic to $\Z$ and
positive Euler characteristic can be properly dominated by
$\C^*\times\C^*$.
\end{proof}

\begin{cor}\label{Cor: No finite map from 2-torus to surface having rank 1 abelian pi_1 and positive Euler characteristic}
No smooth affine surface with an abelian fundamental group of rank $1$ and
positive Euler characteristic can be properly dominated by
$\C^*\times\C^*$.
\end{cor}

\begin{proof}
Let $Y$ be such a surface. Since $\pi_1(Y)$ is a finitely generated abelian
group of rank $1$, using the structure theorem, we have $\pi_1(Y)\cong \Z\oplus T$, where $T$ is a finite abelian group. Passing to the finite \'etale covering corresponding to the subgroup $\Z \hookrightarrow \Z\oplus T\cong \pi_1(Y)$ of finite index $|T|$, we obtain a smooth affine surface $Y'$ with $\pi_1(Y')\cong\Z$.
Moreover, $e(Y')=|T|\,e(Y)>0$, as $e(Y)>0$. If $Y$ were properly dominated by $\C^*\times\C^*$, then so would $Y'$ using the earlier argument (see the proof of Proposition \ref{Prop: Affine surface properly dominated by 2-torus and with finite pi_1}). This contradicts Proposition~\ref{Prop: No finite map from 2-torus to surface having infinite cyclic pi_1 and positive Euler characteristic}. This completes the proof.
\end{proof}

We are now in a position to complete the classification of smooth affine
surfaces of logarithmic Kodaira dimension zero that are properly dominated by
$\C^*\times\C^*$. Combining the preceding propositions with Kojima's
classification of strongly minimal affine surfaces of logarithmic Kodaira
dimension zero yields the following theorem.

\begin{thm}\label{Thm: Affine surfaces with kappa=0 properly dominated by 2-torus}
Let $Y$ be a smooth affine surface properly dominated by
$\C^*\times\C^*$ with $\bar\kappa(Y)=0$. Then $Y$ is isomorphic to one of the
following:
\begin{enumerate}[\rm(1)]
\item $\C^*\times\C^*$;
\item Fujita's surface $H[-1,0,-1]$.
\end{enumerate}
\end{thm}

\begin{proof}
Let
$f\colon X\to Y$
be a finite surjective morphism, where
$X=\C^*\times\C^*$.

If $\deg(f)=1$, then $f$ is a finite birational morphism between smooth affine
surfaces. Hence, by Zariski's Main Theorem, $f$ is an isomorphism. Therefore,
$Y\cong\C^*\times\C^*$, and we are done. Thus, we may assume that
$\deg(f)>1$.

By Remark~\ref{Rem: pi_1 level map is surjective}, we may factor $f$ as
\[
\begin{tikzcd}[column sep=3em,row sep=2em]
&
Z \arrow[dr,"g"] &
\\
X \arrow[ur,"h"] \arrow[rr,"f"'] &&
Y
\\[-5em]
&
{\scalebox{2}{$\circlearrowright$}}
\end{tikzcd}
\]
where $g$ is a finite \'etale morphism,
$h$ is a finite surjective morphism,
$f=g\circ h$, and
$h_\ast\colon\pi_1(X)\to\pi_1(Z)$
is surjective.

Since $g$ is a finite \'etale covering, $\pi_1(Z)$ is finite if and only if
$\pi_1(Y)$ is finite. The latter is impossible by
Proposition~\ref{Prop: Affine surface properly dominated by 2-torus and with finite pi_1}.
Hence $\pi_1(Z)$ is infinite.

Let $Z'$ be a strongly minimal affine surface obtained from $Z$ by repeatedly
contracting exceptional curves. Since logarithmic Kodaira dimension is
preserved under such contractions,
$\bar\kappa(Z')=\bar\kappa(Z)=0$.
Therefore, by Kojima's classification
\cite[Table~1]{Koj1999},
the surface $Z'$ is one of
\[
O(4,1),\quad
O(2,2),\quad
O(1,1,1),\quad
H[-1,0,-1],\quad
H[0,0].
\]

Since the open immersion
$Z'\hookrightarrow Z$
induces a surjective homomorphism on fundamental groups, $\pi_1(Z')$
is also infinite.

We first determine the intermediate surface $Z$. For this, the key step is to identify its strongly minimal model $Z'$.

\medskip
\noindent\textbf{Claim.}
\emph{$Z'=Z\cong\C^*\times\C^*$.}

\smallskip
\noindent\emph{Proof of the claim.}
Since $\pi_1(Z')\twoheadrightarrow\pi_1(Z)$ and $\pi_1(Z)$ is infinite,
the surface $Z'$ must have infinite fundamental group. By
\cite[Table~1]{Koj1999}, it follows that $Z'$ is one of
\[
O(4,1),\quad
O(2,2),\quad
O(1,1,1),\quad
H[-1,0,-1],\quad
H[0,0].
\]

Suppose that $Z'$ is one of
$O(4,1)$,
$O(2,2)$,
or
$H[0,0]$.
Then $\pi_1(Z')\cong\Z$ (see, \cite[Table~1]{Koj1999}). Since
$\pi_1(Z')\twoheadrightarrow\pi_1(Z)$, and $\pi_1(Z)$ is infinite, we obtain $\pi_1(Z)\cong\Z$.
Moreover,
\[
e(Z)=e(Z')+e(Z\setminus Z').
\]
By \cite[Lemma~1.4]{Koj1999}, the complement
$Z\setminus Z'$
is a (possibly empty) disjoint union of affine lines. Hence
$e(Z)\ge e(Z')>0$,
contradicting
Proposition~\ref{Prop: No finite map from 2-torus to surface having infinite cyclic pi_1 and positive Euler characteristic}.
Therefore,
$Z'$ is isomorphic to either
$O(1,1,1)$
or
$H[-1,0,-1]$.

We now show that $Z=Z'$.
Suppose, to the contrary, that $e(Z)>0$.
Since
$h_\ast\colon\pi_1(X)\to\pi_1(Z)$
is surjective and
$\pi_1(X)\cong\Z^2$,
the group $\pi_1(Z)$ is abelian. By Corollary~\ref{Cor: No finite map from 2-torus to surface having rank 1 abelian pi_1 and positive Euler characteristic},
$\pi_1(Z)$ cannot have rank one. Hence
$\pi_1(Z)\cong\Z^2$. Since $e(Z)>0$, we have $b_2(Z)\ge2$.
On the other hand, by Lemma \ref{Lem: Finite morphism basic properties}(2), the finite surjective morphism
$h\colon X\to Z$
implies
$b_2(Z)\le b_2(X)=1$,
a contradiction.
Therefore $e(Z)=0$ and thus $e(Z')=0$ too. Since
$e(Z)\ge e(Z')$
and equality holds if and only if
$Z\setminus Z'=\varnothing$,
by \cite[Lemma~1.4]{Koj1999},
we conclude that
$Z=Z'$.

Finally, from \cite[Table~1]{Koj1999}, it follows that
$\pi_1(H[-1,0,-1])=\langle y, t \mid  yt y^{-1} = t^{-1}\rangle$, which is a nonabelian group, whereas
$\pi_1(Z)$ is abelian as we observed above due to
$h_\ast$ being surjective. Hence $Z$ cannot be isomorphic to $H[-1,0,-1]$.
Therefore,
\[
Z=Z'\cong O(1,1,1)\cong\C^*\times\C^*,
\]
proving the claim.

\medskip

Having identified $Z$, we now complete the proof. Since $Z\cong\C^*\times\C^*$ is an Eilenberg--MacLane $K(\pi,1)$-space and $g\colon Z\to Y$ is a finite \'etale covering, it follows that $Y$ is also an Eilenberg--MacLane $K(\pi,1)$-space. By \cite[Theorem~5.9]{GGH2023}, the only smooth
affine Eilenberg--MacLane surfaces of logarithmic Kodaira dimension zero are $\C^*\times\C^*$ and $H[-1,0,-1]$. Hence
\[
Y\cong\C^*\times\C^*
\quad\text{or}\quad
Y\cong H[-1,0,-1],
\]
which completes the proof.
\end{proof}

\section*{\bf Acknowledgements}

The author expresses his sincere gratitude to Prof.~R.~V.~Gurjar and Prof.~S.~R.~Gurjar for several helpful discussions. The author acknowledges financial support from the Department of Science and Technology, Government of India, through the INSPIRE Faculty Fellowship (Reference No.: DST/INSPIRE/04/2024/003379).

\section*{\bf Data Availability Statement}

Data sharing does not apply to this article, as no datasets were generated or analysed during the current study.

\section*{\bf Declarations}

The authors declare that they have no competing interests. No additional funding was received for this work other than that acknowledged above.

\bibliographystyle{alpha}
\bibliography{ref}

\end{document}